\documentclass[12pt]{article}

\usepackage[top=0.75in]{geometry}
\usepackage{amsmath,amsthm,amssymb,amsfonts}

\usepackage{enumerate, graphicx}

\newtheorem{thm}{Theorem}

\newtheorem{prop}[thm]{Proposition}
\newtheorem{cor}[thm]{Corollary}

\theoremstyle{definition}
\newtheorem*{defi}{Definition}
\newtheorem*{rem}{Remark}

\begin{document}

\title{Moment Calculus on Ramsey Graph }
\author{Thotsaporn ``Aek'' Thanatipanonda\\
Mahidol University International College \\
Nakhon Pathom, Thailand}
\date{August 26, 2019}

\maketitle

\begin{abstract}
When I did my thesis defense presentation eleven years ago,
I chose to present the subject of Ramsey theory from
the moment calculus perspective. 
I don't think I did too well there (although I passed).
Time has passed and this is the chance to redeem myself.
Here we relate Ramsey numbers, $R(k,k)$, with 
the method of moment calculus
by checking the distribution of numbers of monochromatic complete
subgraph of $k$ vertices in the random graphs. 
We also review Delaporte distribution's connection
that was mentioned in the paper by Robertson, Cipolli and Dascalu.  
\end{abstract}

\section{Introduction to Ramsey Numbers}
Ramsey theory is a fascinating but extremely difficult subject
started by British mathematician Frank Ramsey in the early 1900.
But it was Paul Erd{\"o}s who popularize the field. 
Although he had passed away in 1996, this subject is
still alive and gives rise to many interesting research projects. 
In this article, we will restrict ourselves only to Ramsey graph 
(one of the Super-Six theorems, \cite{Graham}).

\begin{defi}[Ramsey Numbers]
$R(k,l)$ is the smallest number of vertices of complete
graph which each edge colored either red or blue such
that no matter how the edges are colored, it must
contain either (monochromatic) red $K_k$ or blue $K_l.$ 
\end{defi}

 \textbf{Examples:}
\[ R(3,3) = 6, \;\ \;\ R(4,4)=18,    \;\ \;\   43 \leq R(5,5) \leq 49, \;\ \;\ 102 \leq R(6,6) \leq 165.  \]

The exact numbers of $R(k,k), \;\ k \geq 3$ are very hard to determine
because of the gigantic possibilities of edge-colorings,  $2^{\binom{n}{2}}$
ways to color edges.  

\noindent The asymptotic behavior of $R(k,k)$
is a famous open problem in combinatorics.
There is even a monetary prize of \$250 for the solution.
Paul Erd\"{o}s used the first moment $E[X]$ to obtain a lower bound of
$\displaystyle \lim_{k \rightarrow \infty}R(k,k) $.

\begin{thm}
\[ \sqrt{2}^k \leq R(k,k) \leq 4^k, \;\ \;\  k \geq 3.\]
\end{thm}

\begin{proof}
For the upper bound: \\ \\
Claim:  $R(m,n) \leq \binom{m+n}{n}.$ 

We show the claim by applying an induction on $m+n$.
First we see that 
\[ R(m,n) \leq R(m-1,n) + R(m,n-1). \]
Then it follows from the induction assumption that 
\[ R(m,n) \leq R(m-1,n) + R(m,n-1) \leq 
\binom{m+n-1}{n} + \binom{m+n-1}{n-1} = \binom{m+n}{n}. \]
The upper bound follows immediately from the claim,
\[ R(k,k) \leq \binom{2k}{k} = \dfrac{(2k)!}{k!k!} 
\approx \dfrac{2^{2k}k^{2k}}{k^k \cdot k^k} = 4^k. \]

For the lower bound (Erd\"{o}s, 1947): \\ \\
We use the property that 
\[ 1-E[X] \geq 0 \rightarrow  P(X=0) > 0. \]
Here we let the random variable $X := X_k$ be the number of 
mono-chromatic subgraph of size $k$
of the complete graph of $n$ vertices.
We want to find $n$ such that $E[X] < 1$ 
then it follows that $R(k,k) > n$
since there must be some random variable (random edge-coloring graph) 
that does not contain monochromatic red $K_k$ or monochromatic blue $K_k$. 

\[ E[X] = \sum X \cdot p(X) 
=\dfrac{(n)_k}{k!}\cdot \dfrac{2}{2^{\binom{k}{2}}} 
\approx \dfrac{n^k}{k!}\cdot \dfrac{2}{2^{\binom{k}{2}}}.\]

\noindent If $n \leq \sqrt{2}^k$ then
\[  E[X] \leq \dfrac{2^{k^2/2} \cdot 2}{k! \cdot 2^{k^2/2-k/2}} = \dfrac{2^{k/2+1}}{k!} < 1
\;\ \;\ \mbox{for } k \geq 3.   \]
\end{proof}

Remark that the precise bound from this idea is
\[  R(k,k) \geq \dfrac{1}{\sqrt{2}e}k2^{k/2}(1+o(1)),  \;\ \;\ k \to \infty. \] 
The bound can be improved by using more sophisticated technique called
\textit{Lovasz local lemma}, see \cite{AS},

\[  R(k,k) \geq \dfrac{\sqrt{2}}{e}k2^{k/2}(1+o(1)),  \;\ \;\ k \to \infty. \] 

\underline{Prize Money Problems} (Ron Graham)
\begin{enumerate}
\item (\$100) Does $\displaystyle \lim_{k \to \infty} R(k,k)^{\frac{1}{k}}$ exist?
\item  (\$250) If the limit exists, what is it?
\end{enumerate}
The idea that Erd\"{o}s used for the lower bound can be extended
by the method called \textit{moment calculus}. 
\section{Moment Calculus of Ramsey Graphs}

Let $S$ be $k$-subsets of $\{1,2,\dots,n\}$ and let $X_S$ be an indicator variable.
\[   X_S = \begin{cases}
1 & \mbox{if subgraph of } K_n \mbox{ induced by } S \mbox{ is monochromatic}  \\ 
0 & otherwise. 
\end{cases}\]

Recall that $X$ is the number of monochromatic $K_k$ 
of the complete edge-coloring graph. Then
\[ X = \sum_S X_S. \]
\subsection*{First moment:}
\begin{align*}
E[X_S] &= \dfrac{2}{2^{\binom{k}{2}}}, \\
E[X] &= \dfrac{2}{2^{\binom{k}{2}}}\cdot \binom{n}{k}.
\end{align*}

\subsection*{Second moment:}
\begin{align*}
E[X^2] &= E\left[(\sum_{S_1}X_{S_1})(\sum_{S_2}X_{S_2})\right]
= \sum_{[S_1,S_2]}E[X_{S_1}X_{S_2}].
\end{align*}

We need to look at how $S_1$ and $S_2$ interact with each other. 

\textbf{Example:} For $k = 3,$
\[ E[X^2] = \dfrac{2}{2^3}\cdot \dfrac{2}{2^3}\cdot\dfrac{(n)_6}{3!3!} 
+ \dfrac{2}{2^3}\cdot \dfrac{2}{2^3}\cdot\dfrac{(n)_5}{2!2!1!}
+ \dfrac{2}{2^5}\cdot\dfrac{(n)_4}{2!1!1!}
+ \dfrac{2}{2^3}\cdot\dfrac{(n)_3}{3!}.\]
For $k = 4,$
\[ E[X^2] = \dfrac{2}{2^6}\cdot \dfrac{2}{2^6}\cdot\dfrac{(n)_8}{4!4!} 
+ \dfrac{2}{2^6}\cdot \dfrac{2}{2^6}\cdot\dfrac{(n)_7}{3!3!1!}
+ \dfrac{2}{2^{11}}\cdot\dfrac{(n)_6}{2!2!2!}
+ \dfrac{2}{2^9}\cdot\dfrac{(n)_5}{3!1!1!}
+  \dfrac{2}{2^6}\cdot\dfrac{(n)_4}{4!} .\]

In fact, we can write the formula of second moment for general $k$
in term of the sum (but not the closed form though). 
\[ E[X^2] = \dfrac{2}{2^{\binom{k}{2}}}\cdot \dfrac{2}{2^{\binom{k}{2}}}
\cdot\dfrac{(n)_{2k}}{k!k!} + \dfrac{2}{2^{\binom{k}{2}}}\cdot \dfrac{2}{2^{\binom{k}{2}}}
\cdot\dfrac{(n)_{2k-1}}{1!(k-1)!(k-1)!} +
\sum_{i=2}^k \dfrac{2}{2^{2\binom{k}{2}-\binom{i}{2}}}
\cdot \dfrac{(n)_{2k-i}}{i!(k-i)!(k-i)!}    . \]
For higher moment with fixed $k$, we need computer to do
the job for us. There are too many ways the objects can
interact with each other. 
The program I wrote can calculate up to the fifth moment for some
small $k.$ 

One nice thing about this calculation is that you can check the correctness
of your formula by comparing the value for each (small $k$) 
with the moment from the different method, i.e.
\[ E[X^r] = \sum_{i=0}^{\infty} i^r P[X=i].\]

\section{Numbers of Monochromatic Complete Subgraphs: 
Normal distribution for ``big $n$'',
Poisson distribution for ``small $n$''}

The following results come from Maple program.

\begin{thm} \label{MM}
The leading term of $E[(X-\mu)^2]$ is 
\[ \dfrac{1}{2}\cdot\dfrac{1}{(k-3)!^2}\cdot\dfrac{n^{2k-3}}{2^{2\binom{k}{2}-2}}.\]
The leading term of $E[(X-\mu)^3]$ is 
\[ \displaystyle \dfrac{1}{(k-3)!^3} \cdot \dfrac{n^{3k-5}}{2^{3\binom{k}{2}-3}}.\]
The leading term of $E[(X-\mu)^4]$ is 
\[ \dfrac{3}{4}\cdot\dfrac{1}{(k-3)!^4}\cdot\dfrac{n^{4k-6}}{2^{4\binom{k}{2}-4}}.\]
The leading term of $E[(X-\mu)^5]$ is 
\[ 5\cdot\dfrac{1}{(k-3)!^5}\cdot\dfrac{n^{5k-8}}{2^{5\binom{k}{2}-5}}.\]

\end{thm}

\noindent With these results, we've already seen an asymptotic normality of $X$ when $n \gg k$.

\begin{cor} \label{normal}
As $k \to \infty$ and $n \geq \dfrac{\sqrt{2}k}{e}2^{\frac{k}{2}}(1+o(1))$, 
the random variable $X$ is normally distributed.   
\end{cor}
Remark: The condition of $n$ is needed to make the leading term significance.

\begin{proof}
We show that the standardized moments $c_m := \dfrac{E[(X-\mu)^m]}{Var^{\frac{m}{2}}}$ 
agree with the coefficients of the moment generating function of standard normal
distribution $e^{\frac{t^2}{2}}$  i.e. $0,1,0,3,0,15,0,105,0,945, \dots$. 
 
\noindent From theorem \ref{MM}, we see
\begin{align*}
c_1 &= \dfrac{0}{\sqrt{Var}} = 0, \\
c_2 &= 1, \\
c_3 &= \dfrac{n^{3k-5}}{n^{3k-9/2}/(2\sqrt{2})} 
= \dfrac{2\sqrt{2}}{\sqrt{n}} \to 0  \;\ \;\ \mbox{as } n \to \infty, \\
c_4 &= \dfrac{3}{4}\cdot 2^2 = 3. \\
c_5 &= \dfrac{5n^{5k-8}}{n^{5k-15/2}/(4\sqrt{2})} 
= \dfrac{20\sqrt{2}}{\sqrt{n}} \to 0  \;\ \;\ \mbox{as } n \to \infty.\\
& \dots
\end{align*}
The other values can be conformed as well.
\end{proof}

The following is a complement of this result: 

In \cite{Poisson}, it was shown that $X_k$ is asymptotically
Poisson as $k \to \infty$ with condition \\
$n \leq \dfrac{\sqrt{2}}{e}k2^{k/2}(1+o(1))$.
That is
\[   P(X_k = j) \approx \dfrac{\lambda^j e^{-\lambda}}{j!}, 
\;\ \;\ \mbox{where } \lambda = \dfrac{\binom{n}{k}}{2^{\binom{k}{2}-1}}. \]
\begin{rem}
It is quite evidence that, asymptotically ($k \to \infty$), 
$X$ does not have Poisson distribution through out.
The leading term of $Var(X)$ is  
$\displaystyle  \dfrac{1}{2(k-3)!^2}\cdot\dfrac{n^{2k-3}}{2^{2\binom{k}{2}-2}}$
and $\dfrac{Var(X)}{E[X]} \sim \dfrac{k(k-1)(k-2)}{(k-3)!}
\cdot\dfrac{n^{k-3}}{2^{\binom{k}{2}}}.$ 
Hence, for a fixed $k, k \geq 4$, and some 
``\textbf{big $n$}'', $n \gg k, \;\ \;\ Var(X) \gg E[X].$ 
(For Poisson distribution, $Var(X) = E[X]$.)
\end{rem}

\subsection{Almost Surely Property of $R(k,k)$}
We will apply Chebyshev's theorem for the almost surely 
property (the set of possible exceptions may be 
non-empty, but it has probability zero) of $R(k,k).$

\begin{thm}[Chebyshev's theorem]
Let $X$ be a non-negative discrete random variable. Then
\[ P(X=0) \leq \dfrac{Var(X)}{E[X]^2}.  \]
\end{thm}

\begin{thm}
For $n \geq \dfrac{\sqrt{2}}{e}k2^{k/2}(1+o(1))$, 
as $k \to \infty, P(X=0) \to 0$ almost surely.
\end{thm}

\begin{proof}
From theorem \ref{MM}, we have
\[  \dfrac{Var(X)}{E[X]^2} \sim \dfrac{k^6}{2n^3}.\]
We then see that   $\dfrac{Var(X)}{E[X]^2} \to 0$.
The result follows from Chebyshev's theorem.
\end{proof}


\section{Delaporte Distribution}
In \cite{Aaron}, the authors found the
best fit for the distribution of $X$,
the number of mono-chromatic subgraph of size $k$
of the complete graph of $n$ vertices, to be Delaporte.
We discuss this distribution in this section.

\begin{defi}[Delaporte distribution] \hfill \\ \\
Let the moment generating function, $mgf(X) 
= \phi(t) := E[e^{tX}].$ 
We define Delaporte distribution by
\[  mgf(D) = \dfrac{e^{\lambda(e^t-1)}}{(1-\beta(e^t-1))^{\alpha}}.   \]
\end{defi}

\noindent The motivation behind this is that $D$ is a convolution of 
a Negative binomial random variable with success probability 
$\dfrac{\beta}{1+\beta}$  and mean $\alpha\beta$ and
a Poisson random variable with mean $\lambda.$ \\

\begin{prop} \label{Del}
The probability mass function of Delaporte distribution is
\[  P(D = j) = \sum_{i=0}^j \dfrac{\Gamma(\alpha+i)}{\Gamma(\alpha)i!}
\left(\dfrac{\beta}{1+\beta}\right)^i \left(\dfrac{1}{1+\beta}\right)^{\alpha}
\dfrac{\lambda^{j-i}e^{-\lambda}}{(j-i)!}.   \]

It also follows that
\begin{align*}
\mu &= E[X] = \lambda+\alpha\beta, \\
Var(X) &= E[(x-\mu)^2] = \lambda+\alpha\beta(1+\beta), \\
E[(X-\mu)^3] &=   \lambda+\alpha\beta(1+3\beta+2\beta^2), \\
E[(X-\mu)^4] &=  3\lambda^2+ \lambda+\alpha\beta(1+\beta)
(3\alpha\beta^2+3\alpha\beta+6\beta^2+6\beta+6\lambda+1),\\
\cdots
\end{align*}
\end{prop}

\begin{proof}
The probability mass function can be calculated from $mgf(D).$ \\ \\
The moment generating function for Poisson:
\begin{align*}
mgf(P) &= \sum_i \dfrac{e^{-\lambda}\lambda^i}{i!}e^{ti} \\
&= e^{-\lambda}  \sum_i \dfrac{(\lambda e^t)^i}{i!} \\
&= e^{\lambda(e^t-1)} .  
\end{align*}
 
\noindent The moment generating function for Negative Binomial:
\begin{align*}
mgf(NB) &= \sum_i \dfrac{\Gamma(\alpha+i)}{\Gamma(\alpha)\cdot i!} 
p^i (1-p)^{\alpha} e^{ti},   \;\ \;\ \mbox{ where } p= \frac{\beta}{1+\beta}\\
&=  \frac{(1-p)^{\alpha}}{[1-pe^t]^{\alpha}} \\
&= \frac{1}{[1+\beta-\beta e^t]^{\alpha}} \\  
&= \frac{1}{[1-\beta (e^t-1)]^{\alpha}} .
\end{align*}

\noindent Hence the $mgf(D)$ is the product of $mgf(P)$ and $mgf(NB).$ 
The probability mass function of $D$ is the convolution of $Pr(P)$ and $Pr(NB).$ \\ \\
The moments in the second part are directly calculated from the
moment generating function.
\end{proof}

\section{Asymptotic/Non-asymptotic fit with Delaporte distribution (?)}

We will discuss the Delaporte distribution 
as the fit of $X$ in three scenarios:
$k\to \infty$ for ``big $n$'', $k\to \infty$ for ``small $n$'' 
and  small $k$.

\subsection{Delaporte Fit as $k\to \infty$ for ``big $n$''}

In this section, we assume $n \geq \dfrac{2k}{e} \cdot 2^{\frac{k}{2}}.$
We already knows that $X$, 
the number of mono. complete subgraphs,
is normally distributed. 
We try to fit the normal distribution of $X$ 
to Delaporte distribution.
We will solve for values of $\lambda, \alpha$ 
and $\beta$ in terms of $n$ and $k$, then calculate 
the moment about the mean of Delaporte distribution
that arises from the parameters $\lambda, \alpha$ and $\beta$ 
to see if it fits. First we solve the parameters 
by matching the leading terms in proposition \ref{Del}
and theorem \ref{MM}:
\begin{align*}
\lambda+\alpha\beta &= \dfrac{1}{k!} \cdot\dfrac{n^k}{2^{\binom{k}{2}-1}},  \\
\alpha\beta^2 &=  \dfrac{1}{2(k-3)!^2}\dfrac{n^{2k-3}}{2^{2\binom{k}{2}-2}},\\
2\alpha\beta^3 &=  \dfrac{1}{(k-3)!^3} \dfrac{n^{3k-5}}{2^{3\binom{k}{2}-3}}.   
\end{align*}
Then we have:
\[ \beta = \dfrac{n^{k-2}}{2^{\binom{k}{2}-1}} \cdot \dfrac{1}{(k-3)!} , \;\ \;\
\alpha = \dfrac{n}{2} ,\] 
\[ \lambda = E[X]-\alpha\beta = \dfrac{n^k}{k! \cdot 2^{\binom{k}{2}-1}}
\left[1-\dfrac{k(k-1)(k-2)}{2n}\right]. \]

\begin{rem}
Assume $n \sim \dfrac{2k}{e} \cdot 2^{\frac{k}{2}},$ we have
\[   \alpha = \dfrac{n}{2} =\dfrac{k}{\sqrt{2}e} \cdot 2^{\frac{k}{2}}, 
\;\  \;\ \beta = \left(\dfrac{k}{e}\right)^{k-2} \cdot \dfrac{2^{\frac{k}{2}}}{4} \cdot
\dfrac{2}{(k-3)!} = \dfrac{k}{e} \cdot \dfrac{2^{\frac{k}{2}}}{2} \cdot
\dfrac{e^3}{\sqrt{2\pi k}} = 
\dfrac{e^2}{2\sqrt{2\pi}} \cdot \sqrt{k} \cdot 2^{\frac{k}{2}} ,\]
 \[  \lambda = \dfrac{2}{\sqrt{2\pi k}}2^{\frac{k}{2}}. \]   

\noindent We see that, with this assumption, $\alpha \gg \beta \gg \lambda,$
as $k \to \infty.$ This verifies the leading terms that we assume earlier.
\end{rem}
With this setting of $\alpha,\beta$ and $\lambda$, the 
Delaporte distribution (proposition \ref{Del}) 
indeed approaches the normal distribution, i.e. the leading terms 
(assume $\alpha \gg \beta \gg \lambda$) of the moments are
\begin{align*}
E[(X-\mu)^2] &\sim \alpha\beta^2 \\
E[(X-\mu)^3] &\sim 2\alpha\beta^3 \\
E[(X-\mu)^4] &\sim 3\alpha^2\beta^4 \\
E[(X-\mu)^5] &\sim 20\alpha^2\beta^5 \\
E[(X-\mu)^6] &\sim 15\alpha^3\beta^6 \\
E[(X-\mu)^7] &\sim 210\alpha^3\beta^7 \\
E[(X-\mu)^8] &\sim 105\alpha^4\beta^8 \\
E[(X-\mu)^9] &\sim 2520\alpha^4\beta^9 \\
E[(X-\mu)^{10}] &\sim 945\alpha^5\beta^{10} \\
\hdots
\end{align*}
We note that the coefficient of $(2k+3)^{th}$ moment is $\dfrac{(2k+3)!}{3(k!)2^k}.$


\subsection{Delaporte Fit as $k\to \infty$ for ``small $n$''}

In \cite{Poisson}, it was shown that $X \sim$ Poisson 
with $\lambda = E[X]= \dfrac{\binom{n}{k}}{2^{\binom{k}{2}}-1}$
under the condition that \[n \leq \dfrac{\sqrt{2}}{e}k2^{k/2}(1+o(1)).\]
In theorem 3 of \cite{Aaron}, the authors claimed, 
under different condition on $n$, that Delaporte distribution
approaches Poisson. The idea is very interesting but 
the statement is confusing (at least to me). 
Here I write my own version of
this theorem using 
the condition on $n$ similar to \cite{Poisson}.

\begin{prop} \label{DP}
If $D \sim Delaporte(\lambda, \alpha, \beta)$, and $P \sim Poisson(\lambda+\alpha\beta)$,
then $mgf(D) \to mgf(P)$ under the assumption: 
\[ \alpha\beta^2 \to 0.\]
\end{prop}

\begin{proof}
We match the term of each moment of Delaporte distribution in proposition \ref{Del}
with the asymptotic distribution of Poisson. Since in Poisson $E[X]=Var(X),$ 
therefore
\[  E_D[X] = \lambda + \alpha\beta \]
and 
\[ Var_D(X) = \lambda + \alpha\beta +\alpha\beta^2\]
must equal. This results to the condition that $\alpha\beta^2 \to 0$. 
Under this condition, the other higher moments of Delaporte fit 
the moments of Poisson perfectly as well.
\end{proof}

\begin{thm}
Given that $n \leq \dfrac{1}{e}k2^{k/2}(1+o(1)).$
Letting
\[ \beta = \dfrac{n^{k-2}}{2^{\binom{k}{2}-1}} \cdot \dfrac{1}{(k-3)!} , \;\ \;\
\alpha = \dfrac{n}{2} ,\] 
\[ \lambda = E[X]-\alpha\beta = \dfrac{n^k}{k! \cdot 2^{\binom{k}{2}-1}}
\left[1-\dfrac{k(k-1)(k-2)}{2n}\right]. \]
If $D \sim Delaporte(\lambda, \alpha, \beta)$, and $P \sim Poisson(\lambda+\alpha\beta)$,
then $mgf(D) \to mgf(P)$ as $k \to \infty.$
\end{thm}

\begin{proof}
For consistence, we define $\beta, \alpha$ and 
$\lambda$ as in the subsection for ``big $n$''. Then
\[  \alpha\beta^2= \dfrac{1}{2(k-3)!^2}\dfrac{n^{2k-3}}{2^{2\binom{k}{2}-2}} 
\leq  \dfrac{e^3}{\pi}\cdot \dfrac{k^2}{2^{k/2}} 
\rightarrow 0 \;\ \;\ \text{ as } k \rightarrow \infty.   \]

We then apply proposition \ref{DP} to conclude the result.
\end{proof}

\subsection{Delaporte fit for small $k$}

In \cite{Aaron}, Robertson successfully fitted the Delaporte($\alpha, \beta, \lambda$)
to the distribution of $X$ (obtained by simulation) for $k= 4,5$ with various $n$.
Parameters $\alpha, \beta, \lambda$ were solved for each $k$ specifically.
The method of moments for a good fit with $\alpha, \beta, \lambda$
does not work well. We could not find the values of these variables
that was mentioned in Robertson's paper.  
There might not be a general methodology for the 
random variable $X$ (for ``small $k$'') to fit to this Delaporte distribution.

\subsection{Conclusion}
The method of moments verifies that, asymptotically,
Delaporte distribution is a good fit for a random variable 
$X$ for both big $n$ and small $n$ cases.


\section*{Appendix: Bonferroni's Inequality}
We discuss Bonferroni's Inequality and its application to
our Poisson and Delaporte distributions.  

The calculations of Bonferroni help us to understand moment calculus better.

\begin{defi}[Moment Generating Function] \hfill \\
\[   G_X(z) = \sum_{i=0}^{\infty} P(X=i)z^i.\]
\end{defi}

\begin{thm}[Inclusion-Exclusion Principle]
\[ P(X=0) = E\left[\binom{X}{0}\right]-E\left[\binom{X}{1}\right]+E\left[\binom{X}{2}\right]-\cdots .  \]
\end{thm}

\begin{proof}
Consider Taylor series expansion about $z=1,$
\[  f(z) = f(1) + \dfrac{f'(1)(z-1)}{1!}+ \dfrac{f''(1)(z-1)^2}{2!}+\dfrac{f'''(1)(z-1)^3}{3!}+\cdots .\]

\noindent Then the moment generating function at $z=0$ becomes 
\[  G_X(0) = G_X(1) - \dfrac{G_X'(1)}{1!}+ \dfrac{G_X''(1)}{2!}-\dfrac{G_X'''(1)}{3!}+\cdots .\]
which implies the statement of theorem.
\end{proof}

\noindent It is still not simple to apply this theorem for $P(X=0).$ 
We might not need exact formula anyway. We only want to use
Bonferroni's inequality to improve the lower bounds. 

\begin{cor}[Bonferroni's inequality:] \hfill \\
For any odd $m$,
\begin{equation} \label{bon}
 P(X=0) \geq \sum_{s = 0}^m(-1)^sE\left[\binom{X}{s}\right], 
\end{equation}

For any even $m$,
\[ P(X=0) \leq \sum_{s = 0}^m(-1)^sE\left[\binom{X}{s}\right]. \]
\end{cor}

Erd\"{o}s used $m =1$ to get the lower bound
$\displaystyle \lim_{k \rightarrow \infty}R(k,k)^{\frac{1}{k}} \geq \sqrt{2}$, i.e.
\[  1-E[X] > 0 \;\  \rightarrow  \;\  P(X=0) > 0. \]

However \eqref{bon} with $m =3,5$ and our moment that we calculated
earlier do not improve the lower bounds of $R(k,k)$ at all.
\[  1-E[X]+E\left[\binom{X}{2}\right]-E\left[\binom{X}{3}\right] > 0 \;\  \rightarrow  \;\  P[X=0] > 0 \]
\[  1-E[X]+E\left[\binom{X}{2}\right]-E\left[\binom{X}{3}\right] 
+E\left[\binom{X}{4}\right]-E\left[\binom{X}{5}\right] > 0 \;\  \rightarrow  \;\  P[X=0] > 0. \]


\subsection*{I: Poisson Paradigm}

Recall the probability mass function of Poisson distribution: 
\[ P(X=j) = \dfrac{\lambda^je^{-\lambda}}{j!} \;\ \;\  \mbox{ for } j \geq 0.\]
\textbf{Crash course in probability} 

Moment generating function:
\[  \phi(t) = \sum_{j=0}^{\infty} P(X=j) t^j 
=\sum_{j=0}^{\infty}  \dfrac{\lambda^je^{-\lambda}}{j!}t^j 
=e^{(t-1)\lambda}\]
and
\[  E[(X)_m] =  \sum_{j=0}^{\infty} P(X=j)(j)_m
= \left. \dfrac{d^m \phi(t)}{dz ^m}\right|_{t=1} = \lambda^m.\]

\noindent We can also verify the inclusion-exclusion principle:

\[ \sum_{s=0}^{\infty} (-1)^sE\left[\binom{X}{s}\right] =
  \sum_{s=0}^{\infty} \dfrac{(-1)^s\lambda^s}{s!} = e^{-\lambda} = P(X=0) .    \]

\noindent Exponential moment generating function:
\[  M_X(t) = E[e^{tX}] = \sum_{j=0}^{\infty} P(X=j) e^{tj} 
=\sum_{j=0}^{\infty}  \dfrac{\lambda^je^{-\lambda}}{j!}e^{tj} 
=e^{\lambda(e^t-1)}\] 
and
\[  E[X^m] =  \sum_{j=0}^{\infty} P(X=j)j^m
= \left. \dfrac{d^m M_X(t)}{{dt}^m}\right|_{t=0} = \sum_{k=0}^m 
S(m,k) \lambda^k,  \]
where $S(m,k)$  is Stirling numbers of the second kind.

\textbf{Crash course on Stirling number}
 
Matrix $s(n,k)$, Stirling number of the first kind, 
and matrix $S(n,k)$, Stirling number of the second kind, 
are inverse of each other. 

Stirling numbers of the first and second kind are dual pair, i.e. 
\[ a_n = \sum_{k=0}^n s(n,k)b_k  \iff b_n = \sum_{k=0}^n S(n,k)a_k. \]

Two examples of these important identity are
\[   (x)_n = \sum_{k=0}^n s(n,k)x^k  \iff  x^n = \sum_{k=0}^n S(n,k)(x)_k ,   \]
and
\[  E[X^m] = \sum_{k=0}^m S(m,k)\lambda^k
\iff \lambda^m = \sum_{k=0}^m s(m,k)E[X^k] = E[(X)_m]   .   \]
As mentioned earlier, this Poisson case 
only valid for ``small $n$''. Therefore it
does not improve the lower bound of $R(k,k).$

\subsection*{II: Delaporte Paradigm}
We verify Bonferroni's Inequality with Delaporte distribution,
``big $n$'' case, that we have done before.

Assume the size of $n \sim \dfrac{2k}{e}2^{\frac{k}{2}}$. We also let
\[ \lambda \sim \dfrac{2^{\frac{k}{2}}}{\sqrt{k}}, \;\   
\alpha \sim k2^{\frac{k}{2}}, \;\ 
\beta \sim \sqrt{k}2^{\frac{k}{2}}. \]
For  each $E[(X)_s]$, \\ 
The first term is $\lambda^s$. \\  
The second term is $s\alpha\beta\lambda^{s-1}$. \\
The third term is $\binom{s}{2}\alpha\beta^2(\alpha+1)\lambda^{s-2}$. \\
The fourth term is $\binom{s}{3}\alpha\beta^3(\alpha+1)(\alpha+2)\lambda^{s-3}$. \\
The fifth term is $\binom{s}{4}\alpha\beta^4(\alpha+1)(\alpha+2)(\alpha+3)\lambda^{s-4}$. \\
$\cdots$ 

Therefore, 
\begin{align*}
 P(X=0) &= \sum_{s=0}^{\infty} (-1)^sE\left[\binom{X}{s}\right]  \\
 &=\sum_{s=0}^{\infty} \dfrac{(-1)^s(\lambda^s+  s\alpha\beta\lambda^{s-1}
 + \binom{s}{2}\alpha\beta^2(\alpha+1)\lambda^{s-2}
 + \binom{s}{3}\alpha\beta^3(\alpha+1)(\alpha+2)\lambda^{s-3}  +\dots )}{s!} \\
 &=\sum_{s=0}^{\infty} (-1)^s(   \dfrac{\lambda^s}{s!}
 +  \dfrac{\alpha\beta}{1!}\dfrac{\lambda^{s-1}}{(s-1)!}
 + \dfrac{\alpha\beta^2(\alpha+1)}{2!}\dfrac{\lambda^{s-2}}{(s-2)!}
 + \dfrac{\alpha\beta^3(\alpha+1)(\alpha+2)}{3!}\dfrac{\lambda^{s-3}}{(s-3)!}  +\dots ) \\
 &= (1 -\alpha\beta + \dfrac{\alpha\beta^2(\alpha+1)}{2!} 
 - \dfrac{\alpha\beta^3(\alpha+1)(\alpha+2)}{3!}
 + \dfrac{\alpha\beta^4(\alpha+1)(\alpha+2)(\alpha+3)}{4!} -\dots
 ) \cdot e^{-\lambda} \\
 &= \dfrac{e^{-\lambda}}{(1+\beta)^{\alpha}} \to 0 ,    
\end{align*}

which agrees which proposition \ref{Del} and is the result we expect.

\end{document}